\documentclass[reqno]{amsart}

\usepackage{amssymb}
\usepackage{amsmath}
\usepackage{color}
\usepackage{graphicx,epsfig}
\usepackage{mathrsfs}
\usepackage{amstext}
\usepackage{latexsym}

\usepackage[labelformat=empty]{caption}
\usepackage{enumerate}

\usepackage{enumitem}
\usepackage[pdftex]{hyperref}
\usepackage[pdftex,open]{bookmark}
\hypersetup{breaklinks=true,
            colorlinks=false,
            citecolor=green,
            urlcolor=blue,
            linkcolor=blue,
            pdfpagemode=UseOutlines,
            pdftitle={Some Consequences Of A Result of Kummer},
            pdfauthor={Omran Kouba}}

\theoremstyle{plain}
\newtheorem{theorem}{Theorem} 
\newtheorem{corollary}[theorem]{Corollary}

\newtheorem{proposition}[theorem]{Proposition}
\theoremstyle{remark}

\newcommand{\reel}{\mathbb{R}}

\newcommand{\ent}{\mathbb{Z}}
\newcommand{\comp}{\mathbb{C}}

\newcommand{\abs}[1]{\left\vert #1\right\vert }

\newcommand{\bg}{\medskip\goodbreak}

\newcommand{\vers}{{\,\longrightarrow\,}}
\newcommand{\egdef}{\buildrel{\scriptscriptstyle{\rm def}}\over{=}}

\newcommand{\Log}{\operatorname{Log}}
\newcommand{\Arg}{\operatorname{Arg}}

\begin{document}
\title[Kummer's Theorem in action]
{Generalized Stieltjes constants
	and integrals  involving the log-log function: Kummer's Theorem in action.}
\author[Omran Kouba]{Omran Kouba$^\dag$}
\address{Department of Mathematics \\
Higher Institute for Applied Sciences and Technology\\
P.O. Box 31983, Damascus, Syria.}
\email{\href{mailto:omran_kouba@hiast.edu.sy}{omran\_kouba@hiast.edu.sy}}
\thanks{$^\dag$ Department of Mathematics, Higher Institute for Applied Sciences and Technology.}

\begin{abstract}
In this note, 
we recall Kummer's Fourier series expansion of the 1-periodic  function that coincides with the logarithm of the Gamma function on the unit interval $(0,1)$, and we use it to find closed forms for some numerical series related to the generalized Stieltjes constants, and some integrals involving the function $x\mapsto \ln \ln(1/x)$.\par

\end{abstract}

\maketitle
\section{\sc Introduction and Notation}\label{sec1}

The aim of this paper is to present an alternative proof of the reflection principle of the first order generalized Stieltjes constants, and to give an alternative approach to the
evaluation of some integrals involving the  function $x\mapsto \ln\ln(1/x)$. The basic tool for this investigation is a result of Kummer recalled below (Theorem \ref{thKum}).

The first order generalized Stieltjes constant $\gamma_1(a)$ is defined for $a\in(0,1)$ by
\[
\gamma_1(a)=\lim_{n\to\infty}\left(\sum_{k=0}^n\frac{\ln(a+k)}{a+k}-\frac{1}{2}\ln^2(n+a)\right).
\]
From this, it is easy to  show that
\[
\gamma_1(a)-\gamma_1(1-a)=\lim_{n\to\infty}\left(\sum_{k=-n}^n\frac{\ln\abs{a+k}}{a+k}\right)\egdef
\sideset{}{'}\sum_{n\in\ent} 
\frac{\ln\abs{a+n}}{a+n},
\]
where the primed sum denotes the ``principal value''  as shown above.
For  integers $p$ and $q$ with $0<p<q$ the difference $\gamma_1(p/q)-\gamma_1(1-p/q)$ can be expressed as follows
\[\gamma_1(p/q)-\gamma_1(1-p/q)
=-\pi \ln(2\pi q e^\gamma) \cot\left(\frac{p\pi}{q}\right) +2\pi\sum_{j=1}^{q-1}
\sin\left(
\frac{2\pi j p}{q}
\right)
\ln\Gamma\left(\frac{ j}{q}\right).
\]
The formula is attributed to Almkvist and Meurman who obtained it by calculating the derivative of the functional equation for the Hurwitz zeta function $\zeta(s,v)$
with respect to $s$ at rational $v$, see \cite{adam}. However, it was
recently discovered  that an equivalent form of this formula
 was already obtained  by Carl Malmsten in 1846 (see \cite{bla}). An elementary proof of this formula will be presented in Proposition
 \ref{pr0}.

In a recent series of articles  (\cite{Moll4},\cite{Moll6},\cite{Moll8},\cite{Moll19},\cite{Var}), the authors  proved some formulas from the \textit{Table of integrals, Series, and Products},  of Gradshteyn and Ryzhik \cite{grad}. Further, the monographs \cite{Moll1,Moll2} are devoted to providing proofs for the formulas in \cite{grad}. In fact, we are particularly interested in integrals involving the  function
$x\mapsto \ln\ln(1/x)$. Indeed, entries $4.325$ of \cite{grad} contain the following evaluations:
\begin{align*}
\int_0^1\frac{\ln(\ln(1/x))}{1+x^2}&=\frac{\pi}{2}\ln\frac{\sqrt{2\pi}\Gamma(3/4)}{\Gamma(1/4)}\\
\int_0^1\frac{\ln(\ln(1/x))}{1+x+x^2}&=\frac{\pi}{\sqrt{3}}\ln\frac{\sqrt[3]{2\pi}\Gamma(2/3)}{\Gamma(1/3)}\\
\int_0^1\frac{\ln(\ln(1/x))}{1+2x\cos t+x^2}&=\frac{\pi}{2\sin t}\ln\frac{(2\pi)^{t/\pi}\Gamma\left(\frac{1}{2}+\frac{t}{2\pi}\right)}{\Gamma\left(\frac{1}{2}-\frac{t}{2\pi}\right)}
\end{align*}

These integrals can be traced back to \cite{Bie}. 
The first of them was the object of a detailed investigation in \cite{Var}, where the author says that his approach can be adapted to prove also the second one. A general approach that yields the first two integrals, and much more evaluations, can also be found in \cite{adam}. This line of investigation was completed by adapting the methods of \cite{Var} to obtain general results that include all the above mentioned integrals in \cite{Moll8}.

Our aim is to present an alternative approach to the evaluation of these integrals.  Our starting point will be  Kummer's Fourier expansion of $\Log \Gamma$, (Theorem \ref{thKum}), where $\Gamma$ is the well-known Eulerian gamma function. This result is attributed to Kummer in (1847), a more accessible reference is \cite[Section 1.7]{and}:
\begin{theorem}[Kummer,\cite{kum}]\label{thKum}
For $0<x<1$,
\[
\ln{\frac{\Gamma(x)}{\sqrt{2\pi}}}
=-\frac{\ln(2\sin(\pi x))}{2}+(\gamma+\ln(2\pi))\left(
\frac{1}{2}-x\right)+\frac{1}{\pi}\sum_{k=1}^\infty
\frac{\ln k}{k}\sin(2\pi k x),
\]
where $\gamma$ is the Euler-Mascheroni  constant.\bg
\end{theorem}

\section{\sc The reflection formula for the first order generalized Stieltjes constants}\label{sec1}

As we explained in the introduction, this formula relates the first order generalized Stieltjes constant $\gamma_1(a)$ to its reflected value $\gamma_1(1-a)$ for rational $a$. The presented proof is different from that
of Almkvist and Meurman, and has the advantage of being elementary in the sense that it does not make use of the
functional equation of the Hurwitz zeta function.

\begin{proposition}\label{pr0}
 for positive integers $p$ and $q$ with $ p< q$, we have
\begin{equation*} \label{pr2}
\sideset{}{'}\sum_{n\in\ent} 
\frac{\ln\abs{n+\frac{p}{q}}}{n+\frac{p}{q}} 
 =-\pi \ln(2\pi q e^\gamma) \cot\left(\frac{p\pi}{q}\right) +2\pi\sum_{j=1}^{q-1}
\sin\left(
\frac{2\pi j p}{q}
\right)
\ln\Gamma\left(\frac{ j}{q}\right).
\end{equation*}
where the primed sum denotes the ``principal value'', defined as follows:
\[
\sideset{}{'}\sum_{n\in\ent}a_n
=\lim_{n\to\infty}\left(\sum_{k=-n}^{ n}a_k\right).
\]
\end{proposition}
\begin{proof}

The statement of Theorem \ref{thKum} is written as

\begin{equation*}\label{EID1}
 \sum_{k=1}^\infty
\frac{\ln k}{k}\sin(2\pi k x)=-
\frac{\pi}{2}\ln \pi+
\frac{\pi}{2}\ln\sin(\pi x)
+\pi \ln(2\pi e^\gamma)\left(x-\frac{1}{2}\right)+\pi\ln\Gamma(x).\tag{1}
\end{equation*}
Now, consider a positive integer $q$ with $q\geq 2$. For $j\in\{1,2,\ldots,q-1\}$ we have

\begin{equation*}\label{EID2}
 \sum_{k=1}^\infty
\frac{\ln k}{k}\sin\left(\frac{2\pi k j}{q}\right)=
- \frac{\pi}{2}\ln \pi+
\frac{\pi}{2}\ln\sin\left(\frac{\pi j}{q}\right)
+\pi \ln(2\pi e^\gamma)\left(\frac{j}{q}-\frac{1}{2}\right)+\pi\ln\Gamma
\left(\frac{ j}{q}\right).\tag{2}
\end{equation*}
Multiply  both sides of \eqref{EID2} by $\sin\left(
\frac{2\pi j p}{q}
\right)$, where $p$ is some integer from $\{1,\ldots,q-1\}$,
and add  the resulting equalities for $j=1,\ldots,q-1$, to obtain
\begin{equation*}\label{EID3}
 \sum_{k=1}^\infty
\frac{\ln k}{k}A_{p,q}(k)=
- \frac{\pi\ln \pi}{2} B_{p,q}+
\frac{\pi}{2}C_{p,q}
+\pi \ln(2\pi e^\gamma)D_{p,q}+\pi\sum_{j=1}^{q-1}
\sin\left(
\frac{2\pi j p}{q}
\right)
\ln\Gamma\left(\frac{ j}{q}\right),\tag{3}
\end{equation*}
where
\begin{align*}
A_{p,q}(k)&=\sum_{j=1}^{q-1}\sin\left(
\frac{2\pi j p}{q}
\right)\sin\left(\frac{2\pi jk}{q}\right)\\
B_{p,q}&=\sum_{j=1}^{q-1}\sin\left(\frac{2\pi j p}{q}\right)\\
C_{p,q}&=\sum_{j=1}^{q-1}\sin\left(\frac{2\pi j p}{q}\right)
\ln\sin\left(\frac{\pi j}{q}\right)\\
D_{p,q}&=\sum_{j=1}^{q-1}\left(\frac{j}{q}-\frac{1}{2}\right)\sin\left(\frac{2\pi j p}{q}\right).
\end{align*}

These sums are now simplified. Let $\omega_q=\exp\left(\frac{2\pi i}{q}\right)$, and use
$\sum_{j=0}^{q-1}\omega_q^{nj}=q\chi_q(n)$ where $\chi_q(n)=1 $ if $n\equiv0\mod q$ and $\chi_q(n)=0 $ otherwise. The imaginary part  of the identity gives \begin{equation*}\label{EID4}
B_{p,q}=0.\tag{4}
\end{equation*}
Also,
\begin{align*}
A_{p,q}(k)&=\frac{1}{2}
\sum_{j=1}^{q-1}\left(\cos\left(
\frac{2\pi j (p-k)}{q}
\right)-\cos\left(\frac{2\pi j(p+k)}{q}\right)\right)\\
&=\frac{1}{2}\Re\left(\sum_{j=0}^{q-1}\omega^{(p-k)j}
-\sum_{j=0}^{q-1}\omega^{(p+k)j}\right)=
\frac{q}{2}(\chi_q(p-k)-\chi_q(p+k)).
\end{align*}
That is
\begin{equation*}\label{EID5}
A_{p,q}(k)=\begin{cases}
\phantom{-}\frac{q}{2}&\text{if $k\equiv p\mod q$},\\
-\frac{q}{2}&\text{if $k\equiv -p\mod q$}.
\end{cases}\tag{5}
\end{equation*}
On the other hand, the change of summation index
$j\leftarrow q-j$ in the formula for $C_{p,q}$ shows that
\[
C_{p,q}
=\sum_{j=1}^{q-1}\sin\left(2\pi p-\frac{2\pi j p}{q}\right)
\ln\sin\left(\pi-\frac{\pi j}{q}\right)=-C_{p,q}.
\]
Thus,
\begin{equation*}\label{EID6}
C_{p,q}=0.\tag{6}
\end{equation*}
Finally, use \eqref{EID4} to obtain
\[
D_{p,q}=\frac{1}{q}\sum_{j=1}^{q-1}j\sin\left(\frac{2\pi j p}{q}\right).
\]
Now for, $0<\theta<\pi$, we have
\begin{align*}
1+2\sum_{j=1}^{q-1}\cos(2j\theta)&=\sum_{j=1-q}^{q-1}e^{2ij\theta}=\frac{e^{2iq\theta}-e^{2i(1-q)\theta}}{e^{2i\theta}-1}\\
&=\frac{\sin((2q-1)\theta)}{\sin\theta}
=\sin(2q\theta)\cot\theta-\cos(2q\theta).
\end{align*}
Taking the derivative with respect to $\theta$ and substituting
$\theta=\pi p/q$ we get
\begin{equation*}\label{EID7}
D_{p,q}=-\frac{1}{2}
 \cot\left(\frac{p\pi}{q}\right).\tag{7}
 \end{equation*}
Replacing \eqref{EID4},\eqref{EID5},\eqref{EID6} and
\eqref{EID7} in \eqref{EID3} we obtain
\begin{align*}\label{EID8}
\sum_{k=0}^\infty\left(
\frac{\ln (qk+p)}{k+p/q}-\frac{\ln (qk+q-p)}{k +1-p/q}\right)
&=-\pi \ln(2\pi e^\gamma) \cot\left(\frac{p\pi}{q}\right)\\
&\phantom{=}{}+2\pi\sum_{j=1}^{q-1}
\sin\left(
\frac{2\pi j p}{q}
\right)
\ln\Gamma\left(\frac{ j}{q}\right).\tag{8}
\end{align*}
The final step is to use the well-known cotangent partial fraction expansion:
\begin{align*}\label{EID9}
\sum_{k=0}^\infty\left(
\frac{1}{k+p/q}-\frac{1}{k +1-p/q}\right)
&=\lim_{N\to\infty}\sum_{k=-N}^N\frac{1}{k+p/q}
=\frac{q}{p}-\sum_{k=1}^N\frac{2p/q}{k^2-(p/q)^2}\\
&=\frac{q}{p}+\sum_{k=1}^\infty\frac{2p/q}{(p/q)^2-k^2}
=\pi\cot\left(\frac{ p\pi}{q}\right).\tag{9}
\end{align*}\nobreak
Thus,
subtracting $(\ln q)$ times \eqref{EID9}  from \eqref{EID8} we obtain the desired conclusion.
\end{proof}

\goodbreak
\noindent\textbf{Examples.} Taking $p=1$ and $q\in\{3, 4\}$ we obtain
\begin{align*}
\sideset{}{'}\sum_{n\in\ent}
\frac{\ln\abs{n+\frac{1}{3}}}{n+\frac{1}{3}} &=
\frac{\pi}{2\sqrt{3}}\ln\left(\frac{3\Gamma^{12}(\frac{1}{3})}{2^8\,\pi^8e^{2\gamma}}\right).\\
\sideset{}{'}\sum_{n\in\ent}
\frac{\ln\abs{n+\frac{1}{4}}}{n+\frac{1}{4}} &=
\pi\ln\left(\frac{\Gamma^{4}(\frac{1}{4})}{2^4\pi^3e^{\gamma}}\right).
\end{align*}

\section{\sc The Evaluation of some integrals involving the 
log-log function}\label{sec2}

In this section we use Theorem \ref{thKum}, to evaluate some difficult integrals.

\begin{proposition}\label{pr3}
For $0<x<1$, we have:
\begin{equation*}
 \int_0^1\frac{\ln \ln(1/u)}{u^2-2(\cos 2\pi x)u+1 }\,du=
 \frac{\pi}{2\sin(2\pi x)}\left((1-2x)\ln(2\pi)+ \ln\left(\frac{ \Gamma(1-x)}{ \Gamma(x)}\right)\right).
\end{equation*}
And, taking the limit as $x$ tend to $1/2$, we obtain
\begin{equation*}
 \int_0^1\frac{\ln \ln(1/u)}{(u+1)^2}\,du=\ln\sqrt{2\pi}+\frac{\Gamma'(1/2)}{2\Gamma(1/2)}
 =\ln\sqrt{\frac{\pi}{2}}-\frac{\gamma}{2}.
\end{equation*}
\end{proposition}
\begin{proof}
Indeed, subtracting the corresponding Kummer's Formulas, for $
\ln \Gamma(x)$ and $\ln \Gamma(1-x)$ we see that, for $0<x<1$ we have
\begin{equation*}
 \ln\left({\frac{\Gamma(x)}{\Gamma(1-x)}}\right)
=(\gamma+\ln(2\pi))\left(
1-2x\right)+\frac{2}{\pi}\sum_{k=1}^\infty
\frac{\ln k}{k}\sin(2\pi k x),\tag{10}
\end{equation*}
or equivalently,
\begin{equation*}\label{EID11}
 \ln\left({\frac{(2\pi)^{x}\Gamma(x)}{(2\pi)^{1-x}\Gamma(1-x)}}\right)
= \gamma \left(
1-2x\right)+\frac{2}{\pi}\sum_{k=1}^\infty
\frac{\ln k}{k}\sin(2\pi k x).\tag{11}
\end{equation*}
Now, using the fact that for $\Re{s>0}$ and $k\geq 1$ we have $\frac{\Gamma(s)}{k^s}=\int_0^\infty t^{s-1}e^{-kt}\,dt$, we conclude that
for $s>0$, and $0<x<1$,  we have
\begin{align*}
\sum_{k=1}^\infty\frac{e^{2\pi i k x}}{k^s}&=
\frac{1}{\Gamma(s)}\sum_{k=1}^\infty\left(\int_0^\infty t^{s-1} e^{-k t} e^{2\pi i k x}dt\right)\\
&=\frac{1}{\Gamma(s)}\int_0^\infty\frac{e^{-t+2\pi i x}}{1-e^{-t+2\pi i x}}\,t^{s-1}\,dt.
\end{align*}
Restricting our attention to the imaginary parts we get

\begin{equation*}
\sum_{k=1}^\infty\frac{\sin({2\pi k x})}{k^s}
=\frac{\sin(2\pi x)}{\Gamma(s)}
\int_0^\infty\frac{t^{s-1}}{e^t+e^{-t}-2\cos(2\pi x)}\,dt.\tag{12}
\end{equation*}
Now, taking the derivative with respect to $s$ at $s=1$ we obtain, for $0<x<1$,
the following:
\begin{align*}
\sum_{k=1}^\infty\frac{\ln k}{k}\sin({2\pi k x})
&= \frac{\Gamma'(1)}{\Gamma^2(1)}
\int_0^\infty\frac{\sin(2\pi x)}{e^{t}+e^{-t}-2\cos(2\pi x) } \, dt\\
&\quad{}- \frac{\sin(2\pi x)}{\Gamma(1)}
\int_0^\infty\frac{\ln t}{e^t+e^{-t}-2\cos(2\pi x)}\,dt.\tag{13}
\end{align*}
Taking into account the facts  $\Gamma'(1)=-\gamma$, $\Gamma(1)=1$, and 
\[
\int_0^\infty\frac{\sin(2\pi x)}{e^t+e^{-t}-2\cos(2\pi x)}\,dt=\pi\left(\frac{1}{2}-x\right),\quad
\text{for $0<x<1$}
\]
we conclude that
\begin{equation*}
\sum_{k=1}^\infty\frac{\ln k}{k }\sin({2\pi k x})
=-\frac{\gamma\pi }{2}(1-2x )
- \sin(2\pi x)
\int_0^\infty\frac{\ln t}{e^t+e^{-t}-2\cos(2\pi x)}\,dt.\tag{14}
\end{equation*}
The change of variables $t= \ln(1/u)$ yields:

\begin{equation*}\label{EID6x}
\frac{2}{\pi}\sum_{k=1}^\infty\frac{\ln k}{k }\sin({2\pi k x})
+\gamma(1-2x )=
- \frac{2\sin(2\pi x)}{\pi}
\int_0^1\frac{\ln \ln(1/u)}{u^2-2(\cos 2\pi x)u+ 1}\,du.\tag{15}
\end{equation*}
Finally, combining \eqref{EID11} and \eqref{EID6x}  we obtain the desired result. Concerning the limit as $x$ tend to $1/2$, we use the well-known fact that $\frac{\Gamma'(1/2)}{\Gamma(1/2)}=\psi(1/2)=-\gamma-2\ln 2$, (see \cite[6.3.3]{abr}).
\end{proof}

\noindent\textbf{Examples.} Taking $x=1/3$, $x=1/4$ and $x=1/6$ we obtain
\begin{align*} 
 \int_0^1\frac{\ln \ln(1/u)}{u^2+u+1 }\,du&=
 \frac{-\pi}{6\sqrt{3}} \ln\left(
\frac{ 3^3}{4^4\pi^8}\,\Gamma^{12}\left(\frac{1}{3}\right)\right).\label{Ex16}\tag{16}\\
 \int_0^1\frac{\ln \ln(1/u)}{u^2+1 }\,du&=\frac{-\pi}{4} \ln\left(
 \frac{1}{4\pi^3}\,\Gamma^4\left (\frac{1}{4}\right)\right).\label{Ex17}\tag{17}\\
 \int_0^1\frac{\ln \ln(1/u)}{u^2-u+1 }\,du&=\frac{-\pi}{3\sqrt{3}} \ln\left(
\frac{1}{(2\pi)^{5}}\,\Gamma^6 \left(\frac{1}{6}\right)\right),\\
&=\frac{-\pi}{3\sqrt{3}} \ln\left(
\frac{3^3}{2^7\pi^{8}}\,\Gamma^{12} \left(\frac{1}{3}\right)\right).\label{Ex18}\tag{18}
\end{align*}
where we used freely the  duplication, and the reflection formulas for the gamma function \cite[6.1.17 and 6.1.18]{abr}. In particular, we used $\Gamma(\frac{1}{6})=\frac{\sqrt{3/\pi}}{\sqrt[3]{2}}\Gamma^2(\frac{1}{3})$ that follows readily from these formulas.

\bigskip
The second degree polynomial in the integrand's denominator in Proposition \ref{pr3} has negative discriminant. In the next proposition the corresponding denominator has real roots outside the interval $[0,1]$. This case seems to be new to the best knowledge of the author.

\begin{proposition}
Let $A_\Gamma:\reel\vers\reel$ be the function defined by
\[
A_\Gamma(y)=-\frac{\ln(2\pi)}{2}\,y+\frac{\sinh(\pi y)}{\pi} \int_0^1\frac{\ln \ln(1/u)}{u^2+2\cosh (\pi y)\,u+1 }\,du.
\]
Then, for $y\in\reel$ we have
\[
\Gamma\left(\frac{1+iy}{2}\right)=\sqrt{\frac{\pi}{\cosh(\pi y/2)}}
\, e^{iA_\Gamma(y)}.
\]
\end{proposition}
\begin{proof}
Let us rephrase Proposition \ref{pr3}, by taking $x=\frac{t+1}{2}$ in order to give more symmetric aspect to the formula there:

\begin{equation*}
\forall\,t\in(-1,1),\quad \int_0^1\frac{\ln \ln(1/u)}{u^2+2\cos (\pi t)\,u+1 }\,du=
 \frac{ \pi}{2\sin(\pi t)}\left( \ln(2\pi)\,t+ \ln\left(\frac{ \Gamma(\frac{1+t}{2})}{ \Gamma(\frac{1-t}{2})}\right)\right),
\end{equation*}
or equivalently, for $-1<t<1$, we have
\begin{equation*}
\exp\left(-\ln(2\pi) t+\frac{2\sin(\pi t)}{\pi} \int_0^1\frac{\ln \ln(1/u)}{u^2+2\cos (\pi t)\,u+1 }\,du\right)=
 \frac{ \Gamma(\frac{1+t}{2})}{ \Gamma(\frac{1-t}{2})}.
\end{equation*}
Using analytic continuation we deduce that, for $-1<\Re z<1$ we have also
\begin{equation*}
\exp\left(-\ln(2\pi) z+\frac{2\sin(\pi z)}{\pi} \int_0^1\frac{\ln \ln(1/u)}{u^2+2\cos (\pi z)\,u+1 }\,du\right)= 
 \frac{ \Gamma(\frac{1+z}{2})}{ \Gamma(\frac{1-z}{2})}.
\end{equation*}
In particular, setting $z=iy$  with $y\in\reel$, we obtain
\begin{equation*}
e^{2iA_\Gamma(y)}= \frac{ \Gamma(\frac{1+iy}{2})}{ \Gamma(\frac{1-iy}{2})}.
\end{equation*}
But, by  Euler's reflection formula \cite[6.1.17]{abr} we know that 
\[
\abs{\Gamma\left(\frac{1+iy}{2}\right)}^2=
\Gamma\left(\frac{1+iy}{2}\right)
\overline{\Gamma\left(\frac{1+iy}{2}\right)}=
\Gamma\left(\frac{1+iy}{2}\right)
\Gamma\left(\frac{1-iy}{2}\right)=
\frac{\pi}{\cosh(\pi y/2)},
\]
therefore, the square of the continuous function: 
\[
y\mapsto \sqrt{\dfrac{\cosh(\pi y/2)}{\pi}}\Gamma\left(\dfrac{1+iy}{2}\right)e^{-iA_\Gamma(y)}
\]
is equal to $1$ for every $y\in\reel$, hence, it must be constant and consequently identical to $1$ which is its value for $y=0$.
\end{proof}
\begin{corollary}
Let the principal determination of the argument of a nonzero complex number $z$ be denoted by $\Arg$, and
let $\alpha$ be defined by the formula
\[
\alpha=\inf\left\{y>0:\Gamma\left(\dfrac{1+iy}{2}\right)=-\sqrt{\frac{\pi}{\cosh(\pi y/2)}} \right\}.
\]
Then, for every $y\in(-\alpha,\alpha)$ we have
\[
\int_0^1\frac{\ln \ln(1/u)}{u^2+2\cosh (\pi y)\,u+1 }\,du
=\ln\sqrt{2\pi}\cdot\frac{\pi y}{\sinh(\pi y)}+\frac{\pi}{\sinh(\pi y)}\Arg \Gamma\left(\frac{1+iy}{2}\right).
\]
Moreover, using \textsc{Mathematica}${}^{\text{\SMALL\textregistered}}$
Software, we readily obtain
$\alpha\approx 10.106\,689\,535\,698$.
\end{corollary}
\begin{proof}
The definition of $\alpha$ implies that
\[
\forall\,y\in(-\alpha,\alpha),\quad
\Gamma\left(\dfrac{1+iy}{2}\right)\in\comp\setminus((-\infty,0]\times\{0\}).
\]
Thus, the function $y\mapsto \Arg(\Gamma(\frac{1+iy}{2}))-A_\Gamma(y)$ is continuous on $(-\alpha,\alpha)$, takes its values in $2\pi\ent$, and is equal to $0$ for $y=0$. Therefore, $A_\Gamma(y)=\Arg(\Gamma(\frac{1+iy}{2}))$, for every $ y\in(-\alpha,\alpha)$, which is the desired conclusion.
\end{proof} 

\noindent\textbf{Examples.}
\begin{align*}
\int_0^1\frac{\ln \ln(1/u)}{u^2+4u+1 }\,du
&=
\frac{\ln(2\pi)}{\sqrt{3}}\,\ln\left(\frac{1+\sqrt{3}}{\sqrt{2}}\right)+\frac{\pi}{\sqrt{3}}\Arg\Gamma\left(\frac{1}{2}+
\frac{i}{\pi}\ln\left(\frac{1+\sqrt{3}}{\sqrt{2}}\right)\right).
\\
\int_0^1\frac{\ln \ln(1/u)}{u^2+3u+1 }\,du
&=\frac{2\ln(2\pi)}{\sqrt{5}}\,\ln(\phi)+\frac{2\pi}{\sqrt{5}}\Arg \Gamma\left(\frac{1}{2}+\frac{i}{\pi}\ln(\phi)\right) .
\end{align*}
where $\phi=\dfrac{1+\sqrt{5}}{2}$ is the golden ratio. 

\noindent More generally, for $2<k<2\cosh(\alpha\pi)\approx 6.156\times 10^{13}$, the following holds
\[
\int_0^1\frac{\ln \ln(1/u)}{u^2+ku+1 }\,du
=\frac{2\ln(2\pi)}{\sqrt{k^2-4}}\,\ln(\phi_k)+\frac{2\pi}{\sqrt{k^2-4}}\Arg \Gamma\left(\frac{1}{2}+\frac{i}{\pi}\ln(\phi_k)\right) 
\]
with $\phi_k=\dfrac{\sqrt{k+2}+\sqrt{k-2}}{2}$.

It is worth mentioning that \textsc{Mathematica}${}^{\text{\SMALL\textregistered}}$~10.4 gives the results of examples
\eqref{Ex16}, \eqref{Ex17} and \eqref{Ex18}, but it fails to give the results
of the previous examples. However, numerical quadrature confirms the results.

\bigskip
In our final proposition we consider the evaluation
of another $\log$-$\log$ integral.   
This integral was given in \cite{adam} as a corollary of a more difficult evaluation. Our approach is straightforward and simpler. 

\begin{proposition}[\cite{adam}]
For any complex number $z$ with $\Re z>0$, we have
\[
F(z)\egdef\int_0^1\frac{t^{z-1}\ln(\ln(1/t))}{1+t^z}dt
=-\frac{\ln 2}{2z}\Log(2z^2)
\] 
where $\Log$ is the principal branch of the logarithm.
\end{proposition}
\begin{proof}
We start by evaluating $F(1)$. Note that
\[
F(1)=\int_0^1\frac{ \ln(\ln(1/t))}{1+t}dt
=\int_0^\infty\frac{ e^{-x}}{1+e^{-x}} \ln(x) \,dx.
\]
So,
\[
\abs{F(1)-\sum_{k=1}^{n}(-1)^{k-1}
\int_0^\infty e^{-kx}\ln(x)dx}\leq\int_0^\infty\frac{ \abs{\ln x}}{1+e^{x}}\,e^{-nx}\,dx.
\]
Because $x\mapsto \frac{ \abs{\ln x}}{1+e^{x}}$ is integrable on $(0,+\infty)$, we conclude using Lebesgue's dominated convergence theorem that
\[
\lim_{n\to\infty}\int_0^\infty\frac{ \abs{\ln x}}{1+e^{x}}\,e^{-nx}\,dx=0
\]
Thus,
\[
F(1)=\sum_{k=1}^\infty (-1)^{k-1}\int_0^\infty e^{-kx}\ln(x)dx.
\]
A simple change of variables shows that
\[
\int_0^\infty e^{-kx}\ln(x)dx=\frac{1}{k}
\int_0^\infty e^{-u}(\ln u-\ln k)\,du=
-\frac{\gamma}{k}-\frac{\ln k}{k}
\]
since $\int_0^\infty\ln(u)e^{-u}\,du=\Gamma'(1)=-\gamma$. It follows that
\begin{equation*}\label{pr41}
F(1)=-\gamma \sum_{k=1}^\infty\frac{(-1)^{k-1}}{k}
+\sum_{k=1}^\infty\frac{(-1)^{k}\ln k}{k}
=-\gamma\ln 2+\sum_{k=1}^\infty\frac{(-1)^{k}\ln k}{k}.\tag{20}
\end{equation*}
Now, note that
\begin{align*}
\ln^2(k+1)-\ln^2k
&=\ln^2k\left(\left(1+\frac{1}{\ln k}\ln\left(1+\frac{1}{k}\right)\right)^2-1\right)\\
&=\ln^2k\left(\frac{2}{k\ln k}+\mathcal{O}\left(\frac{1}{k^2\ln k}\right)\right)\\
&=\frac{2\ln k}{k}+\mathcal{O}\left(\frac{\ln k}{k^2}\right).
\end{align*}
This proves that the series $\sum\left(\ln^2(k+1)-\ln^2k-2\frac{\ln k}{k}\right)$
is convergent. Consequently, if we define
$G_n=\sum_{k=1}^n\frac{\ln k}{k}$ then
there is a real number $\ell$ such that
$G_n=\frac{1}{2}\ln^2n+\ell+o(1)$. But
\begin{align*}
\sum_{k=1}^{2n}\frac{(-1)^k\ln(k)}{k}&=
\sum_{k=1}^{n}\frac{\ln(2k)}{k}-\sum_{k=1}^{2n}\frac{\ln k }{k} 
=(\ln 2) H_n+G_n-G_{2n}\\
&=(\ln 2)(\ln n+\gamma)+\frac{1}{2}\left(\ln^2(n)-\ln^2(2n)\right)+o(1)\\
&=-\frac{1}{2}\ln^22+\gamma\ln 2+o(1),
\end{align*}
where we used $H_n=\sum_{k=1}^n1/k= \ln n+\gamma+o(1)$, (see \cite[4.1.32]{abr}).

\noindent Now, let $n$ tend to $+\infty$ to obtain
\[
\sum_{k=1}^{\infty}\frac{(-1)^k\ln(k)}{k}=-\frac{1}{2}\ln^22+\gamma\ln 2.
\]
Combining this with \eqref{pr41} we conclude that
$F(1)=-\frac{1}{2}\ln^22$.

Next, for $z\in(0,+\infty)$ the change of variables $t^z=u$ shows that
\[
F(z)=\frac{1}{z}\int_0^1\frac{\ln\ln(1/u)-\ln z)}{1+u}\,du
=\frac{F(1)}{z}-\frac{\ln(z)\ln(2)}{z}=-
\frac{\ln(2z^2)\ln(2)}{2z},
\]
and the desired conclusion follows by analytic continuation.
\end{proof}

\bigskip
\textsc{Acknowledgement.} The author would like to thank the anonymous reviewers for their comments that greatly improved the manuscript.

\end{document}